\documentclass[11pt]{article}
\usepackage{amssymb}
\usepackage{amsmath}
\usepackage{amsthm}
\usepackage[textwidth=16cm, textheight=24cm]{geometry}

\usepackage[english]{babel}
\usepackage[utf8]{inputenc}
\usepackage[T1]{fontenc}
\newtheorem{thm}{Theorem}
\newtheorem{remark}[thm]{Remark}
\newtheorem{prop}[thm]{Proposition}
\newtheorem{cor}[thm]{Corollary}
\newtheorem{lem}[thm]{Lemma}
\newtheorem{defn}[thm]{Definition}

\def\lin{\operatorname {lin}}
\def\dim{\operatorname{dim}}

    \def\tensor{\otimes}

    \def\hook{\lrcorner}

\begin{document}

    \title{An upper bound for the Waring rank of a form}
    \author{Joachim Jelisiejew\thanks{Supported by the project
    ``Secant varieties, computational complexity, and toric degenerations''
    realised within the Homing Plus programme of Foundation
    for Polish Science,
    co-financed from European Union, Regional
    Development Fund. This paper is a part of ``Computational complexity,
    generalised Waring type problems and tensor decompositions'' project
    within ``Canaletto'',  the executive program for scientific and
    technological cooperation between Italy and Poland, 2013-2015.}}
    \maketitle

    \begin{abstract}
        In this paper we introduce the open Waring rank of a form of degree
        $d$ in $n$ variables and prove the that this rank in bounded from
        above by
        \[
        \binom{n+d-2}{d-1} - \binom{n+d-6}{d-3}
        \]
        whenever $n, d\geq 3$. This proves the same upper bound for the
        classical Waring rank of a form, improving the result of \cite{BBS}
        and giving, as far as we know, the best upper bound known.
    \end{abstract}

    {\footnotesize\noindent\textbf{e-mail address:}
    \texttt{jj277546@students.mimuw.edu.pl}.\\
    \textbf{keywords:} sums of powers of linear polynomials, Waring problem,
    Waring rank.}

    \section{Introduction}
    \def\pp#1{\left( #1 \right)}
    \def\DRk#1{Ork\pp{#1}}
    \def\DRkF{\DRk{F}}
    \def\DRknd{\DRk{n,d}}
    \def\DSnd{S(n,d)}

    In \cite{BBS} the authors introduce a stronger version of the Waring rank
    of a homogeneous polynomial and give an upper bound for this rank.
    Denoting this rank by $S(n, d)$, the key point of the proof is the inequality
    \begin{equation}\label{eqref:inductivestep:eqn}
        S(n, d) \leq S(n, d-1) + S(n-1, d),
    \end{equation}
    for $n,d\geq 3$, which gives a recursive step in the proof of the bound $S(n, d) \leq
    \binom{n + d - 2}{d - 1}$. The base cases of the recursion are the
    equalities 
    \begin{equation}\label{eqref:basecases}
        S(2, d) = d\quad \mbox{and} \quad S(n, 2) = n,
    \end{equation}
    so the smallest case where the
    obtained upper bound may be sharp is $S(3, 3)$; \cite{BBS} gives the
    bound $S(3, 3) \leq 6$.
    In this article we introduce an even slightly stronger version of the rank,
    denoted $\DRknd$.
    We prove the inequality
    \eqref{eqref:inductivestep:eqn} together with the base cases \eqref{eqref:basecases} for $\DRknd$
     thus obtaining a bound
    \[
    \DRknd \leq \binom{n+d-2}{d-1}.
    \]
    Next we prove that $\DRk{3, 3} = 5$ which improves the upper bound 
    to
    \[
    \DRknd \leq \binom{n+d-2}{d-1} - \binom{n+d-6}{d-3}.
    \]
    In the proof we will both adopt (in Lemma \ref{goodforms:lem}) and reference the ideas from Johannes
    Kleppe Master thesis \cite{Kl}.

    \vspace{1em}
    \def\AA{\mathbb{A}}%
    \paragraph{Notation.} Let $k$ be an algebraically closed field of characteristic $0$ and $S = k[x_1,\dots,x_n]$ be a polynomial ring. We will often think of
    $S_1$ as an affine space; in this spirit let $V \subsetneq S_1$ be
    a Zariski closed subset. Let $S^* = k\left[
    \frac{\partial}{\partial x_1},\dots, \frac{\partial}{\partial x_n}
    \right]$ be the ring of differential operators with its usual action on
    $S$, which is denoted by $(-) \hook (-): S^* \tensor S \to S$. For $F\in S_d$ by $F^{\perp} \subseteq S^*$ we denote the annihilator
    of $F$ with respect to this action.

    \begin{defn}
        A form $F\in S_d$ \emph{essentially depends on $n$ variables} if
        it cannot be written using less than $n$ variables after a linear
        change of coordinates.
    \end{defn}

    \begin{defn}
        For a form $F$ of degree $d$ in $S$ and $V \subseteq S_1$ let $m = \DRk{F, V}$ be the minimal natural number such that there exists
        a presentation
        \[
        F = \sum_{i=1}^m l_i^d,\mbox{ where } l_i\notin V.
        \]
        or $\DRk{F, V} = \infty$ if such presentation does not exist.
        Define the \emph{open Waring rank} of $F$ by
        \[\DRkF := \sup \left\{\DRk{F, V}\ |\ V \subsetneq S_1\mbox{ homogeneous
        and Zariski closed}\right\}.\]
        Finally take
        \[
        \DRknd := \sup\left\{ \DRkF\ |\ F\in S_d\mbox{ essentially depends
        on }n\mbox{ variables}\right\}.
        \]
    \end{defn}
    \begin{remark} The classical Waring rank of a form $F$ is equal to
        $\DRk{F, \emptyset}$. The rank defined in \cite{BBS} is similar to the
        one defined above, the difference is that the authors consider only subsets $V \subsetneq S_1$ which are finite sums
    of hyperplanes through the origin. Thus we have an inequality $\DRknd \geq S(n,
    d)$.
\end{remark}

    The paper is divided into two sections, preceded by a preliminary part. In
    the first section we prove:
    \begin{thm}\label{ref:mainthm:thm}
        Let $n, d\geq 2$ be integers.
        We have equations $\DRk{2, d} = d$ and
        $\DRk{n, 2} = n$. Moreover
        \[
        \DRk{n, d} \leq \DRk{n-1, d} + \DRk{n, d-1}
        \]
        for every $n, d \geq 3$.
    \end{thm}
    The proof is a copy of the proof of
    \eqref{eqref:inductivestep:eqn} and \eqref{eqref:basecases} from \cite{BBS}. Unfortunately the proof
    given there is, formally, just a special case of the proof required and
    moreover Bia\l{}ynicki-Birula and Schinzel are also concerned with
    non-homogeneous polynomials which makes their proof more
    complicated.

    In the second part we prove the following theorem, with an
    immediate corollary bounding the open Waring rank:
    \begin{thm}\label{ref:ineqthm:thm}
        $\DRk{3, 3} = 5$.
    \end{thm}
    \begin{cor}\label{ref:improvedbound:cor}
        Let $n, d\geq 3$ be integers, then
        \[
        \DRk{n, d} \leq \binom{n+d-2}{d-1} - \binom{n+d-6}{d-3}.
        \]
    \end{cor}

    \section{Preliminaries}

    Let us recall a well known lemma, whose proof can be found e.g. in
    \cite{Kl}
    \begin{lem}\label{ref:containment:lem}
        If $F\in S_d$ and $\bigcap_{i=1}^k \mathfrak{m}_i \subseteq F^{\perp}$, where
        $\mathfrak{m}_i$ are homogeneous ideals of distinct points
        $[l_i]\in\mathbb{P}S_1$, then
        \[
        F\in \lin\left(l_1^d, l_2^d, \dots, l_k^d\right).
        \]
    \end{lem}

    This lemma will be used together with a special case of unmixedness property of
    complete intersections in $\mathbb{P}^2$:
    \begin{lem}\label{ref:unmixedness:lem}
        If two homogeneous forms $G, H$ intersect transversally in points
        $\{a_1,\dots,a_k\} \subseteq \mathbb{P}^2$, then $\bigcap_i
        \mathfrak{m}_{a_i} = (G, H)$.
    \end{lem}

    Another lemma, whose proof can be found in \cite{Kl}, is concerned with
    linear systems obtained from the apolar ideal of a form:
    \begin{lem}\label{ref:badispower:lem}
        Let $F$ be a form in $S_d$. Choose $e\leq d$ and consider $\mathcal{L} = (F^{\perp})_e$ as a
        linear system on $\mathbb{P}S_1$.
        A point $[l]\in \mathbb{P}S_1$ is a base point of $\mathcal{L}$ if and
        only if there exists a differential $\partial\in S^{*}_{d-e}$ such
        that $\partial F = l^{e}$.
    \end{lem}
    \begin{proof}[Sketch of proof]
        Fix a point $[l]\in \mathbb{P}S_1$ with homogeneous ideal
        $\mathfrak{m}_l$.
        The point $[l]$ is a base point of $\mathcal{L}$ iff $(\mathfrak{m}_l)_{e}
        \supseteq (F^{\perp})_{e} = \bigcap\left\{(\partial F)^{\perp}_e \ |\ \partial \in
        S^{*}_{d-e}\right\}$ iff there exists $\partial \in
        S^{*}_{d-e}$ such that $l^e = \partial  F$.
    \end{proof}

    \begin{cor}\label{ref:badareclosed:lem}
        Fix $d \geq e\geq 1$. Denote by $Ess_{d,e}$ the set of forms $F\in S_d$ such
        that no nonzero element of $S_{e}^{*}$ annihilates $F$.
        The set of $F\in Ess_{d,e}$ such that
        $(F^{\perp})_{d-e}$ has a base
        point in $\mathbb{P}S_1$ is closed in $Ess_{d, e}$.
    \end{cor}
    \begin{proof}
        Note that $Ess_{d,e}$ is Zariski open in $S_d$. Denote by $W$ the subset of forms $F\in Ess_{d,e}$ such that
        $(F^{\perp})_{d-e}$ has a base point in $\mathbb{P}S_1$.
        Consider the closed subvariety
        \[
        \left\{ (F, [\partial], [l])\in Ess_{d, e} \times
        \mathbb{P}S^{*}_{e} \times \mathbb{P}S_1\ |\ l^{d-e}\mbox{ and }\partial F\mbox{ are linearly dependent}
        \right\}.
        \]
        The projection to the first coordinate gives the set of forms $F\in
        Ess_{d, e}$
        such that there exist $\partial\in S^*_e,\ l\in S_1$ and $\lambda,
        \lambda'\in k$, not both equal zero, satisfying $\lambda
        l^{d-e} = \lambda' \partial F$. As $l^{d-e}\neq 0$ and $\partial F\neq
        0$ from the definition of $Ess_{d, e}$ we have $\lambda\lambda'\neq
        0$, which is equivalent, by Lemma \ref{ref:badispower:lem}, to $F\in W$.
    \end{proof}

    \section{Proof of Theorem \ref{ref:mainthm:thm}}

    The proof will be divided into three independent lemmas.

    \begin{lem}
        Let $d\geq 2$ be an integer, then $\DRk{2, d} = d$.
    \end{lem}

    \begin{proof}
        \def\DFp{F^{\perp}}
        Let $S = k[x_1, x_2], F\in S_d$ and $V \subsetneq S_1$ be homogeneous and
        Zariski-closed. We would like to
        prove that $\DRk{F, V}\leq d$.
        It is a classical result by Sylvester that $\DFp$ is a
        complete intersection generated by elements of degrees $d_1$ and $d_2$ such
        that $d_1 + d_2 = d+2$. If $\min(d_1, d_2) = 1$ then $F$ is
        does not essentially depend on two variables. Thus $\min(d_1, d_2) \geq
        2$ and $\max(d_1, d_2) \leq d$, in particular the linear system
        $\DFp_d$ on $\mathbb{P}S_1$ is base point free.
        By Bertini Theorem \cite[Thm III.10.9]{Har} a general element $D$ of
        $\DFp_{d}$ is smooth and does not intersect $V$. The zero set
        of $D$ is a sum of $d$ points, which, by Lemma
        \ref{ref:containment:lem}, gives a required presentation of $F$.
        On the other hand the apolar ideal $(x_1^{d-1}x_2)^{\perp}$ is generated by a square of
        a linear form  and a $d$-th power of another linear form, thus there are no
        smooth forms of degree less than $d$ in this ideal and so the Waring
        rank of $x_1^{d-1}x_2$ is $d$.
    \end{proof}

    \begin{lem}\label{ref:quadratic:lem}
        Let $n\geq 1$ be an integer, then $\DRk{n, 2} = n$.
    \end{lem}

    \begin{proof}
        The inequality $\DRk{n, 2} \geq n$ is trivial because the sum of less
        that $n$ squares does not essentially depend on $n$ variables.
        We prove the other inequality by induction on $n$, the base being
        clear.
        Let $n\geq 2$. Take $F\in S_2$ which essentially depends on $n$
        variables and $V \subsetneq S_1$ homogeneous and Zariski-closed.

        \def\DTone{S^*_1}
        Think about $\DTone$ as an affine space. For
        $\partial\in \DTone$ the condition $\partial^2\hook F = 0$ is
        Zariski-closed. Take any $\alpha\in \DTone$ such that
        $\alpha^2\hook F \neq 0$ and $V(\alpha) \not\subset V$.
        Let
        \[
        F' = F - \frac{\pp{\alpha\hook F}^2}{2\cdot\alpha^2\hook F},
        \]
        then $\alpha\hook F' = 0$, thus $F'$ may be written, after a linear
        change of coordinates, in $n - 1$ variables $x'_1,\dots,x_{n-1}'$ such
        that $\alpha\hook x'_1 = 0$. From the definition of $F'$ it follows that $F$
        may be written using one more variable than $F'$, thus $F'$
        essentially depends on $n-1$ variables. Furthermore $V' = V\cap
        V(\alpha)\neq V(\alpha)$ is a homogeneous Zariski-closed set, so that
        $\DRk{F', V'} \leq n -1$ by induction, and we obtain $\DRk{F, V} \leq
        \DRk{F', V'} + 1\leq n$.
    \end{proof}

    \begin{lem}
        Let $n, d\geq 3$ be integers, then
        \[
        \DRk{n, d} \leq \DRk{n-1, d} + \DRk{n, d-1}.
        \]
    \end{lem}

    \begin{proof}
        \def\DTone{S^*_1}
        Take $F\in S_d$ which essentially depends on $n$ variables and $V
        \subsetneq S_1$ homogeneous and Zariski-closed. Take $\alpha\in \DTone$ such
        that $V(\alpha) \not\subset V$ and $F' = \alpha\hook F$ essentially
        depends on $n$ variables (these are open non-empty conditions).
        The form $F'$ has a presentation
        \[
        F' = \sum_{i=1}^m l_i^{d-1},
        \]
        where $m = \DRk{F', V \cup V(\alpha)}$ and $l_i\not\in V\cup
        V(\alpha)$. Note that $l_i\not\in V(\alpha)$ is equivalent to
        $\alpha\hook l_i \neq 0$. Take
        \begin{equation}\label{eqref:localF1}
            F_1 = \sum_{i=1}^m \alpha_i\cdot l_i^d
        \end{equation}
        where $\alpha_i = (d\cdot\alpha\hook l_i)^{-1}$, then $\alpha\hook (F - F_1)
        = 0$. Let $T \subseteq \left\{ 1, 2, \dots, m \right\}$ be a minimal
        set of indexes such that there exists $0\neq \beta = \beta_T\in \DTone$ such that
        \begin{enumerate}
            \item $V(\beta) \not\subset V$,
            \item $F_2 := F - \sum_{i\in T} \alpha_i\cdot l_i^d$ is annihilated by
                $\beta$,
        \end{enumerate}
        \def\DtmpFtwop{\pp{F_{2}}^{\perp}_1}
        (the set $T = \{1, \dots, m\}$ with $\alpha = \beta_T$ satisfies the
        above hypotheses except, perhaps, minimality).
        We claim that the form $F_2$ obtained from a minimal $T$ essentially depends
        on $n - 1$ variables.
        If this is not the case then we take $i\in T$ such that $F_2 + \alpha_i\cdot
        l_i^d$ essentially depends on more variables than $F_2$. The space $\DtmpFtwop$ is at
        least two-dimensional, thus its intersection with $\pp{l_i^d}^{\perp}_1$
        contains a non-zero element $\beta'$. Since $l_i\in V(\beta')\setminus V$,
        we have $V(\beta')\not\subset V$ and the set $T' := T \setminus \{i\}$
        satisfies the above conditions. This contradicts the minimality of $T$.

        Since $F_2\in k[x_1,\dots,x_n]$ essentially depends on $n-1$ variables
        lying in $V(\beta)$ and $V \cap V(\beta) \neq
        V(\beta)$, the form $F_2$ may be written as $m_2 \leq
        \DRk{n-1, d}$ powers of linear forms taken from outside $V$. The field $k$ is algebraically
        closed, thus  \eqref{eqref:localF1} shows that $F = (F-F_2) + F_2$ may
        be written using at most $m + m_2 \leq \DRk{n, d-1} + \DRk{n-1, d}$ powers of
        linear forms taken from outside $V$.
    \end{proof}

    \section{Proof of Theorem \ref{ref:ineqthm:thm}}

    From now on $n = 3$,~i.e.~$S := k[x_1, x_2, x_3]$.
    First we deal with the majority of forms, using the following lemma:

    \begin{lem}\label{goodforms:lem}
        Let $F\in S_3$ be such that $V( (F^{\perp})_2) \subseteq
        \mathbb{P}S_1$ is an empty set. Then
        $\DRk{F, V}\leq 4$ for any homogeneous closed $V \subsetneq S_1$.
    \end{lem}

    \begin{proof}
        Let $V' \subseteq \mathbb{P}S_1$ be the image of
        $V\setminus\left\{0\right\}$, then $V'$ is closed and not equal to
        $\mathbb{P}S_1$.

        By Bertini theorem \cite[Thm III.10.9]{Har} applied to the base point free linear system
        $(F^{\perp})_2$ on $\mathbb{P}S_1$ we see that the general element $D$ of this system is
        smooth. At the same time a general
        element $D$ intersects $V'$ properly i.e. $\dim V(D)\cap V' < \dim
        V'$. We choose $D_0$ satisfying both properties.

        Restricting to $V(D_0)$ and using Bertini theorem once more we obtain
        an element $D_1\in (F^{\perp})_2$ such that $V(D_0)\cap V(D_1)$ is
        smooth of dimension zero and $V(D_0)\cap V(D_1)\cap V'$ is empty.
        From Lemmas \ref{ref:containment:lem},\ref{ref:unmixedness:lem} it follows that
        \[
        F\in \lin\left(l_{a_1}^3, l_{a_2}^3, l_{a_3}^3, l_{a_4}^3\right)
        \]
        where $\{a_1,a_2,a_3,a_4\} = V(D_0, D_1)$ so
        $\{a_1,a_2,a_3,a_4\}\cap V' = \emptyset$.
    \end{proof}

    Now we would like to show that the set of ``bad forms'', i.e.~those which
    do not satisfy the assumptions of Lemma \ref{goodforms:lem},
    is closed in the (open) set of all forms which essentially depend on
    three variables.

    \begin{cor}\label{closed:cor}
        Denote by $Ess$ the (open) set of forms which essentially depend on three
        variables and let $W \subseteq Ess$ be the subset consisting of forms such that $V( (F^{\perp})_2) \subseteq
        \mathbb{P}S_1$ is not an empty set. Then $W$ is closed in $Ess$.
    \end{cor}

    \begin{proof}
        This follows from  Corollary \ref{ref:badareclosed:lem} applied to the case
        $d=3,e=1$.
    \end{proof}

    Finally,we need an explicit characterisation of the
    ``bad forms'' due to Kleppe:

    \begin{prop}\label{Kleppes:lemma}
        Consider the set of forms $F\in S_3$ essentially dependent on three variables and such that $V( (F^{\perp})_2) \subseteq
        \mathbb{P}S_1$ is \emph{not} an empty set. Every element of this set is an image, under a linear
        change of basis in $S_1$, of one of the following forms
        \begin{equation}\label{badforms:standardform}
        x_0 x_1^2 + x_1 x_2^2\quad\mbox{ or }\quad x_0^3 + g\mbox{ where } g\in
        k[x_1,x_2]_3.
        \end{equation}
        Furthermore the classical Waring rank of $x_0 x_1^2 + x_1 x_2^2$ is
        five.
    \end{prop}
    \begin{proof}
        See \cite[Theorem 2.3]{Kl}.
    \end{proof}

    \vspace{1cm}
    \begin{proof}[Proof of Theorem \ref{ref:ineqthm:thm}]
        By Proposition \ref{Kleppes:lemma} it sufficies to prove $\DRk{3,
        3}\leq 5$.
        Take a form $F\in S_3$ which essentially depends on three variables and a homogeneous closed subset $V
        \subsetneq \mathbb{A}^3$.

        If $F$ satisfies the assumptions of Lemma \ref{goodforms:lem} then
        $\DRk{F, V}\leq 4$ and we are done. Denote the set of the forms
        which essentially depend on three variables and satisfy the
        assumptions of Lemma \ref{goodforms:lem} by $U$.

        If $F\notin U$, then $F\in W$, where $W$ was defined in
        Corollary \ref{closed:cor}. In this case we would like
        to find a linear form $l$ such that $F + l^3\in U$.
        After a linear change of
        coordinates we can assume $F$ is of the form from Lemma
        \ref{Kleppes:lemma}. For $x_0x_1^2 + x_1x_2^2$ the
        form $l^3 = (x_0+x_1)^3$ will do and in the second case we can write $g =
        x_1x_2(a_1x_1 + a_2x_2)$ where $a_1\neq 0$, then $l^3 = (x_0+x_2)^3$ will
        do.

        The set of forms
        which essentially depend on three variables is open in the set of all
        forms and the set $U$ is open in this set by Corollary
        \ref{closed:cor}, so that $U$ is open in the
        set of all forms. We have just seen that $U$ has non-empty intersection
        with $\{F + l^3\}$, so $U\cap \{F + l^3\}$ is open in this set,
        choosing $l\notin V$ such that $F + l^3\in U$ we get the required
        result.
    \end{proof}

    \subsection*{Acknowledgements}

    The author is very grateful to Jaros\l{}aw Buczy\'nski for suggesting
    this research topic as well as his constant support, and to Enrico
    Carlini for insisting on writing down
    this paper.

\end{document}